\documentclass[reqno, 12pt]{amsart}
 \usepackage[a4paper, hmargin = 1.25 in, vmargin = 1.5 in]{geometry}
 \usepackage{amsmath, amssymb, amsfonts}
 \usepackage{amsthm}
 
 \newtheorem{thm}{Theorem}[section]

 \newtheorem{lemma}[thm]{Lemma}
 
 \theoremstyle{definition}
 \newtheorem{definition}{Definition}[section]
 \theoremstyle{remark}
 
 \newtheorem{example}{Example}
 \numberwithin{equation}{section}

\begin{document}
\begin{center}
\begin{title}
\title{\bf\Large{{A Remark on the Memory Property of Fractional Difference Operators}}}
\end{title}

\vskip 0.18 in

\begin{author}
\author {Jagan Mohan Jonnalagadda\footnote[1]{Department of Mathematics, Birla Institute of Technology and Science Pilani, Hyderabad - 500078, Telangana, India. email: {j.jaganmohan@hotmail.com}}}
\end{author}
\end{center}

\vskip 0.18 in

\noindent{\bf Abstract:} Fractional difference operators possess nonlocal structure which largely affects and complicates the qualitative analysis of fractional difference equations. In this article, we discuss the effect of this memory property on asymptotic behaviour of solutions of nabla fractional difference equations.

\vskip 0.18 in

\noindent{\bf Key Words:} Nabla fractional difference, nonlocal, memory, solution, asymptotic behaviour.

\vskip 0.2 in

\noindent{\bf AMS Classification:} 26A33, 34A08, 39A12, 39A70.

\vskip 0.2 in

\section{Introduction \& Preliminaries}
The notions of fractional integrals and derivatives \cite{Ki, Po} date back to the works of Euler but the concepts of nabla fractional sum and differences \cite{Go} are just a decade old. Nabla fractional calculus also represents a natural instrument to model nonlocal phenomena either in space or time. The nabla fractional sum or difference of any function contains information about the function at earlier points, so it possesses a long memory effect. From science to engineering, there are many processes that involve different space or time scales. In many problems of this context, the dynamics of the system can be formulated by fractional difference equations which include the nonlocal effects.

Throughout the article, we shall use the following notations, definitions, and known results of nabla fractional calculus \cite{Go}. For any real number $a$, denote by $\mathbb{N}_{a}  = \{a, a + 1, a + 2, \ldots\}$. The backward jump operator $\rho : \mathbb{N}_{a + 1} \rightarrow \mathbb{N}_{a}$ is defined by $$\rho(t) = t - 1, \quad t \in \mathbb{N}_{a + 1}.$$ Define the $\mu^{th}$-order nabla fractional Taylor monomial by $$H_{\mu}(t, a) = \frac{(t - a)^{\overline{\mu}}}{\Gamma(\mu + 1)} = \frac{\Gamma(t - a + \mu)}{\Gamma(t - a)\Gamma(\mu + 1)}, \quad t \in \mathbb{N}_{a}, \quad \mu \in \mathbb{R} \setminus \{\ldots, -2, -1\},$$ provided the right-hand side exists. Here $\Gamma(\cdot)$ denotes the Euler gamma function. Observe that $H_{\mu}(a, a) = 0$ and $H_{\mu}(t, a) = 0$ for all $\mu \in \{\ldots, -2, -1\}$ and $t \in \mathbb{N}_{a}$. From \cite{Ji}, we have that 
\begin{equation} \label{As}
\lim_{t \rightarrow \infty} H_{\mu - 1}(t, a) = 0, \quad 0 < \mu < 1.
\end{equation}
Let $u: \mathbb{N}_{a} \rightarrow \mathbb{R}$ and $N \in \mathbb{N}_1$. The first order backward (nabla) difference of $u$ is defined by $$\big{(}\nabla u\big{)}(t) = u(t) - u(t - 1), \quad t \in \mathbb{N}_{a + 1},$$ and the $N^{th}$-order nabla difference of $u$ is defined recursively by $$\big{(}{\nabla}^{N}u\big{)}(t) = \Big{(}\nabla\big{(}\nabla^{N - 1}u\big{)}\Big{)}(t), \quad t\in \mathbb{N}_{a + N}.$$

\begin{definition}[See \cite{Go}] \label{RL S}
Let $u: \mathbb{N}_{a + 1} \rightarrow \mathbb{R}$ and $\nu > 0$. The $\nu^{\text{th}}$-order nabla sum of $u$ based at $a$ is given by
\begin{equation}
\nonumber \big{(}\nabla ^{-\nu}_{a}u\big{)}(t) = \sum^{t}_{s = a + 1}H_{\nu - 1}(t, \rho(s))u(s), \quad t \in \mathbb{N}_{a},
\end{equation}
where by convention $\big{(}\nabla ^{-\nu}_{a}u\big{)}(a) = 0$ and $\nabla ^{-0}_{a}u = u$. 
\end{definition}

\begin{definition}[See \cite{Go}] \label{RL D}
Let $u: \mathbb{N}_{a + 1} \rightarrow \mathbb{R}$, $\nu > 0$ and choose $N \in \mathbb{N}_1$ such that $N - 1 < \nu \leq N$. The $\nu^{\text{th}}$-order {\it Riemann--Liouville} nabla difference of $u$ based at $a$ is given by
\begin{equation}
\nonumber \big{(}\nabla ^{\nu}_{a}u\big{)}(t) = \Big{(}\nabla^N\big{(}\nabla_{a}^{-(N - \nu)}u\big{)}\Big{)}(t), \quad t\in\mathbb{N}_{a + N}.
\end{equation}
\end{definition}

\begin{thm}[See \cite{Ah}] \label{Ah}
Let $u: \mathbb{N}_{a + 1} \rightarrow \mathbb{R}$, $\nu > 0$ and choose $N \in \mathbb{N}_1$ such that $N - 1 < \nu < N$. Then
\begin{equation}
\nonumber \big{(}\nabla ^{\nu}_{a}u\big{)}(t) = \sum^{t}_{s = a + 1}H_{-\nu - 1}(t, \rho(s))u(s), \quad t\in\mathbb{N}_{a + 1}.
\end{equation}
\end{thm}
Clearly, $\big{(}\nabla ^{\nu}_{a}u\big{)}(t)$ depends on the values of $u$ at the points $a + 1$, $a + 2$, $\dots$, $t$ unlike the first order nabla difference $\big{(}\nabla u\big{)}(t)$, which just depends on the values of $u$ at the points $t - 1$ and $t$ only. Hereafter, we say that the memory of $\big{(}\nabla ^{\nu}_{a}u\big{)}(t)$ is $(t - a)$.

\begin{lemma} \cite{Go} \label{Power}
Let $\nu$, $\mu \in \mathbb{R}$ and $n \in \mathbb{N}_{1}$ such that $N - 1 < \nu < N$. Then, $$\nabla^{\nu}_{a}H_{\mu}(t, a) = H_{\mu - \nu}(t, a), \quad t \in \mathbb{N}_{a + 1}.$$
\end{lemma}

\section{Examples}
The memory property of nabla fractional difference operators play an important role in describing the behaviour of solutions of the corresponding nabla fractional difference equations. We demonstrate this fact through the following examples.
\begin{example}
Assume $0 < \nu, \mu < 1$. Consider the first order and the $\nu^{\text{th}}$-order nabla difference equations
\begin{align}
& \big{(}\nabla u\big{)}(t) = H_{\mu - 1}(t, \rho(a)), \quad t \in \mathbb{N}_{a + 1}, \label{D 31} \\
& \big{(}\nabla^{\nu}_{\rho(a)} u\big{)}(t) = H_{\mu - 1}(t, \rho(a)), \quad t \in \mathbb{N}_{a + 1}. \label{D 32}
\end{align} 
For $u_{0} \in \mathbb{R}$, every solution of \eqref{D 31} is of the form
\begin{equation}
u(t, a, u_{0}) = u_{0} + H_{\mu}(t, \rho(a)), \quad t \in \mathbb{N}_{a},
\end{equation}
which tends to $\infty$ as $t \rightarrow \infty$ \cite{Bo, El}. Every solution of \eqref{D 32} is of the form
\begin{equation}
u(t, a, u_{0}) = u_{0} H_{\nu - 1}(t, \rho(a)) + H_{\nu + \mu - 1}(t, \rho(a)), \quad t \in \mathbb{N}_{a}.
\end{equation}
Assume $0 < \nu + \mu < 1$. Then, from \eqref{As}, every solution of \eqref{D 32} tends to 0 as $t \rightarrow \infty$. 
\end{example}

\begin{example}
Assume $0 < \nu < 1$, $c : \mathbb{N}_{a + 1} \rightarrow \mathbb{R}$ such that $1 - c(t) \neq 0$ for $t \in \mathbb{N}_{a + 1}$. Consider the first order and the $\nu^{\text{th}}$-order nabla difference equations
\begin{align}
& \big{(}\nabla u\big{)}(t) = c(t) u(t), \quad t \in \mathbb{N}_{a + 1}, \label{D 51} \\
& \big{(}\nabla^{\nu}_{\rho(a)} u\big{)}(t) = c(t) u(t), \quad t \in \mathbb{N}_{a + 1}. \label{D 52}
\end{align} 
From \cite{Bo, El}, we know that every solution of \eqref{D 51} tends to $0$ as $t \rightarrow \infty$ if $$\lim_{t \rightarrow \infty}\left|\prod^{t}_{s = a + 1}\left[\frac{1}{1 - c(s)}\right]\right| = 0.$$ Also, it follows from \cite{Wu H} that every solution of \eqref{D 52} tends to 0 as $t \rightarrow \infty$ if $\left|1 - c(t)\right| \geq 1$ for $t \in \mathbb{N}_{a + 1}$. For example, if we take $c(t) = 2$ for $t \in \mathbb{N}_{1}$, every solution of $$\big{(}\nabla u\big{)}(t) = 2 u(t), \quad t \in \mathbb{N}_{1},$$ doesn't tends to zero as $t \rightarrow \infty$ whereas every solution of $$\big{(}\nabla^{\nu}_{\rho(0)} u\big{)}(t) = 2 u(t), \quad t \in \mathbb{N}_{1},$$ tends to 0 as $t \rightarrow \infty$. 
\end{example}

Motivated by these facts, in this article, we study the effect of this memory property on asymptotic behaviour of solutions of the following nabla fractional difference equation:
\begin{equation}
\big{(}\nabla^{\nu}_{\rho(a)} u\big{)}(t) = c(t) u(t - 1), \quad t \in \mathbb{N}_{a + 1}, \label{L 1} 
\end{equation} 
where $c : \mathbb{N}_{a + 1} \rightarrow \mathbb{R}$ and $0 < \nu < 1$. Finally, we compare these properties with that of the following first order nabla difference equation:
\begin{equation}
\big{(}\nabla u\big{)}(t) = c(t) u(t - 1), \quad t \in \mathbb{N}_{a + 1}. \label{L 2} 
\end{equation} 

\section{Main Results}
Define a sequence recursively as follows: $\tilde{E}_{c, \nu}(a, a) = 1$ and 
\begin{equation} \label{E}
\tilde{E}_{c, \nu}(t, a) = c(t) \tilde{E}_{c, \nu}(t - 1, a) - \sum^{t - 1}_{s = a}H_{-\nu - 1}(t, \rho(s))\tilde{E}_{c, \nu}(s, a), \quad t\in\mathbb{N}_{a + 1}.
\end{equation}

\begin{thm} \label{T 1}
Assume $c : \mathbb{N}_{a + 1} \rightarrow \mathbb{R}$ and $N - 1 < \nu < N$, $N \in \mathbb{N}_{1}$. Then, the unique solution of the nabla fractional initial value problem
\begin{equation} \label{IVP}
\begin{cases}
\big{(}\nabla^{\nu}_{\rho(a)} u\big{)}(t) = c(t) u(t - 1), \quad t \in \mathbb{N}_{a + 1}, \\
u(a) \in \mathbb{R},
\end{cases}
\end{equation}
is given by
\begin{equation} \label{Sol}
u(t) = u(a) \tilde{E}_{c, \nu}(t, a), \quad t \in \mathbb{N}_{a},
\end{equation}
where $\tilde{E}_{c, \nu}(t, a)$ is defined by \eqref{E}.
\end{thm}

\begin{proof}
From Theorem \ref{Ah}, we have
\begin{align} \label{1}
\nonumber \big{(}\nabla ^{\nu}_{\rho(a)}u\big{)}(t) & = \sum^{t}_{s = a}H_{-\nu - 1}(t, \rho(s))u(s) \\ & = \sum^{t - 1}_{s = a}H_{-\nu - 1}(t, \rho(s))u(s) + u(t), \quad t \in \mathbb{N}_{a + 1}.
\end{align}
Then, it follows from \eqref{IVP} and \eqref{1} that
\begin{equation} \label{2}
u(t) = c(t) u(t - 1) - \sum^{t - 1}_{s = a}H_{-\nu - 1}(t, \rho(s))u(s), \quad t \in \mathbb{N}_{a + 1}.
\end{equation}
Take $t = a + n$, $n \in \mathbb{N}_{1}$. From \eqref{Sol}, it is enough to prove that the statement 
\begin{equation} \label{Sol 1}
u(a + n) = u(a) \tilde{E}_{c, \nu}(a + n, a), \quad n \in \mathbb{N}_{1},
\end{equation}
is true by using the principle of strong mathematical induction on $n$. From \eqref{E}, \eqref{1} and \eqref{2}, we obtain
\begin{align*}
u(a + 1) & = c(a + 1) u(a) - \sum^{a}_{s = a}H_{-\nu - 1}(a + 1, \rho(s))u(s) \\ & = u(a) \left[c(a + 1) - \sum^{a}_{s = a}H_{-\nu - 1}(a + 1, \rho(s)) \tilde{E}_{c, \nu}(s, a)\right] \\ & = u(a) \left[c(a + 1)\tilde{E}_{c, \nu}(a, a) - \sum^{a}_{s = a}H_{-\nu - 1}(a + 1, \rho(s)) \tilde{E}_{c, \nu}(s, a)\right] \\ & = u(a)\left[c(a + 1) \tilde{E}_{c, \nu}(a + 1 - 1, a) - \sum^{a + 1 - 1}_{s = a}H_{-\nu - 1}(a + 1, \rho(s))\tilde{E}_{c, \nu}(s, a)\right] \\ & = u(a) \tilde{E}_{c, \nu}(a + 1, a).
\end{align*}
Thus, the statement \eqref{Sol 1} is true for $n = 1$. Again, from \eqref{E}, \eqref{1} and \eqref{2}, we obtain
\begin{align*}
u(a + 2) & = c(a + 2) u(a + 1) - \sum^{a + 1}_{s = a}H_{-\nu - 1}(a + 2, \rho(s))u(s) \\ & = u(a) \left[c(a + 2) \tilde{E}_{c, \nu}(a + 1, a) - \sum^{a + 1}_{s = a}H_{-\nu - 1}(a + 2, \rho(s))\tilde{E}_{c, \nu}(s, a)\right] \\ & = u(a)\left[c(a + 2) \tilde{E}_{c, \nu}(a + 2 - 1, a) - \sum^{a + 2 - 1}_{s = a}H_{-\nu - 1}(a + 2, \rho(s))\tilde{E}_{c, \nu}(s, a)\right] \\ & = u(a) \tilde{E}_{c, \nu}(a + 2, a).
\end{align*}
Thus, the statement \eqref{Sol 1} is true for $n = 2$. Assume the statement \eqref{Sol 1} is true for $k = 3, 4, \cdots, n - 1$. That is, 
\begin{equation} \label{Sol 2}
u(a + k) = u(a) \tilde{E}_{c, \nu}(a + k, a), \quad k \in \mathbb{N}^{n - 1}_{1}.
\end{equation}
Now, we prove the statement \eqref{Sol 1} is true for $n$. It follows from \eqref{E}, \eqref{1} and \eqref{2} that
\begin{align*}
u(a + n) & = c(a + n) u(a + n - 1) - \sum^{a + n - 1}_{s = a}H_{-\nu - 1}(a + n, \rho(s))u(s) \\ & = u(a)\left[c(a + n) \tilde{E}_{c, \nu}(a + n - 1, a) - \sum^{a + n - 1}_{s = a}H_{-\nu - 1}(a + n, \rho(s))\tilde{E}_{c, \nu}(s, a)\right] \\ & = u(a) \tilde{E}_{c, \nu}(a + n, a).
\end{align*}
Hence, by the principle of strong mathematical induction on $n$, the statement \eqref{Sol 1} is true for each $n \in \mathbb{N}_{1}$. Since we obtain the solution by the method of steps, $$u(t) = u(a) \tilde{E}_{c, \nu}(t, a), \quad t \in \mathbb{N}_{a},$$ is the unique solution of the initial value problem \eqref{IVP}. The proof is complete.
\end{proof}

\begin{thm} \label{T 2}
Assume $0 < \nu < 1$ and $\left|c(t) + \nu\right| \leq \nu$ for $t \in \mathbb{N}_{a + 1}$. Then, the solutions of 
\begin{equation}
\big{(}\nabla^{\nu}_{\rho(a)} u\big{)}(t) = c(t) u(t - 1), \quad t \in \mathbb{N}_{a + 1},
\end{equation}
satisfy
\begin{equation}
\lim_{t \rightarrow \infty}u(t) = 0.
\end{equation}
\end{thm}

\begin{proof}
We show that 
\begin{equation} \label{AA}
\left|\tilde{E}_{c, \nu}(t, a)\right| \leq  H_{\nu - 1}(t, \rho(a)), \quad t \in \mathbb{N}_{a + 1}.
\end{equation}
Take $t = a + n$, $n \in \mathbb{N}_{1}$. From \eqref{AA}, it is enough to prove that the statement 
\begin{equation} \label{AB}
\left|\tilde{E}_{c, \nu}(a + n, a)\right| \leq  H_{\nu - 1}(a + n, \rho(a)), \quad n \in \mathbb{N}_{1},
\end{equation}
is true by using the principle of strong mathematical induction on $n$. For $n = 1$, from \eqref{E}, we have
\begin{align*}
\left|\tilde{E}_{c, \nu}(a + 1, a)\right| & = \left|c(a + 1) \tilde{E}_{c, \nu}(a, a) - \sum^{a}_{s = a}H_{-\nu - 1}(a + 1, \rho(s))\tilde{E}_{c, \nu}(s, a)\right| \\ & = \left|(c(a + 1) + \nu) \tilde{E}_{c, \nu}(a, a)\right| \\ & = \left|c(a + 1) + \nu\right| \\ & \leq \nu = H_{\nu - 1}(a + 1, \rho(a)).
\end{align*}
Thus, the statement \eqref{AB} is true for $n = 1$. For $n = 2$, from \eqref{E}, we have
\begin{align*}
\left|\tilde{E}_{c, \nu}(a + 2, a)\right| & = \left|c(a + 2) \tilde{E}_{c, \nu}(a + 1, a) - \sum^{a + 1}_{s = a}H_{-\nu - 1}(a + 2, \rho(s))\tilde{E}_{c, \nu}(s, a)\right| \\ & = \left|c(a + 2) \tilde{E}_{c, \nu}(a + 1, a) + \frac{\nu (1 - \nu)}{2}\tilde{E}_{c, \nu}(a, a) + \nu \tilde{E}_{c, \nu}(a + 1, a)\right| \\ & = \left|(c(a + 2) + \nu) \tilde{E}_{c, \nu}(a + 1, a) + \frac{\nu (1 - \nu)}{2}\right| \\ & \leq \left|c(a + 2) + \nu \right| \left|\tilde{E}_{c, \nu}(a + 1, a)\right| + \frac{\nu (1 - \nu)}{2} \\ & \leq \nu^2 + \frac{\nu (1 - \nu)}{2} \\ & = \frac{\nu (\nu + 1)}{2} = H_{\nu - 1}(a + 2, \rho(a)).
\end{align*}
Thus, the statement \eqref{AB} is true for $n = 2$.  Assume the statement \eqref{AB} is true for $k = 3, 4, \cdots, n - 1$. That is, 
\begin{equation} \label{Sol B}
\left|\tilde{E}_{c, \nu}(a + k, a)\right| \leq H_{\nu - 1}(a + k, \rho(a)), \quad k \in \mathbb{N}^{n - 1}_{1}.
\end{equation}
Now, we prove the statement \eqref{AB} is true for $n$. Consider
\begin{align} \label{L}
\nonumber & \left|\tilde{E}_{c, \nu}(a + n, a)\right| \\ \nonumber & = \left|c(a + n) \tilde{E}_{c, \nu}(a + n - 1, a) - \sum^{a + n - 1}_{s = a}H_{-\nu - 1}(a + n, \rho(s))\tilde{E}_{c, \nu}(s, a)\right| \\ & = \left|c(a + n) \tilde{E}_{c, \nu}(a + n - 1, a) - \sum^{n - 1}_{k = 0}H_{-\nu - 1}(a + n, \rho(a + k))\tilde{E}_{c, \nu}(a + k, a)\right|.
\end{align}
Observe that 
\begin{align*}
- H_{-\nu - 1}(a + n, \rho(s)) & = - \frac{1}{\Gamma(-\nu)}(a + n - s + 1)^{\overline{-\nu - 1}} \\ & = - \frac{\Gamma(a + n - s - \nu)}{\Gamma(a + n - s + 1)\Gamma(-\nu)} \\ & = \nu  \frac{\Gamma(a + n - s - \nu)}{\Gamma(a + n - s + 1)\Gamma(1 - \nu)}
\end{align*}
Since $0 < \nu < 1$ and $a \leq s \leq a + n - 1$, we have $1 - \nu > 0$, $a + n - s + 1 > 0$ and $a + n - s - \nu > 0$, implying that $$- H_{-\nu - 1}(a + n, \rho(s)) > 0.$$ Then, from \eqref{L} and \eqref{Sol B}, we have
\begin{align} 
\nonumber & \left|\tilde{E}_{c, \nu}(a + n, a)\right| \\ \nonumber & = \Big{|}c(a + n) \tilde{E}_{c, \nu}(a + n - 1, a) + \nu \tilde{E}_{c, \nu}(a + n - 1, a) \\ \nonumber & \quad - \sum^{n - 2}_{k = 0}H_{-\nu - 1}(n, \rho(k))\tilde{E}_{c, \nu}(a + k, a)\Big{|} \\ \nonumber & = \left|(c(a + n) + \nu) \tilde{E}_{c, \nu}(a + n - 1, a) - \sum^{n - 2}_{k = 0}H_{-\nu - 1}(n, \rho(k))\tilde{E}_{c, \nu}(a + k, a)\right| \\ \nonumber & \leq \left|c(a + n) + \nu\right| \left|\tilde{E}_{c, \nu}(a + n - 1, a)\right| - \sum^{n - 2}_{k = 0}H_{-\nu - 1}(n, \rho(k)) \left|\tilde{E}_{c, \nu}(a + k, a)\right| \\ \nonumber & \leq \nu H_{\nu - 1}(a + n - 1, \rho(a)) - \sum^{n - 2}_{k = 0}H_{-\nu - 1}(n, \rho(k))H_{\nu - 1}(a + k, \rho(a)) \\ \nonumber & = H_{-\nu - 1}(n, \rho(n - 1)) H_{\nu - 1}(a + n - 1, \rho(a)) \\ \nonumber & \quad - \sum^{n - 2}_{k = 0}H_{-\nu - 1}(n, \rho(k))H_{\nu - 1}(a + k, \rho(a)) \\ \nonumber & = - \sum^{n - 1}_{k = 0}H_{-\nu - 1}(n, \rho(k))H_{\nu - 1}(a + k, \rho(a)) \\ \nonumber & = - \sum^{n}_{k = 0}H_{-\nu - 1}(n, \rho(k))H_{\nu - 1}(a + k, \rho(a)) + H_{-\nu - 1}(n, \rho(n))H_{\nu - 1}(a + n, \rho(a)) \\ \nonumber & = - \sum^{n}_{k = 0}H_{-\nu - 1}(n, \rho(k))H_{\nu - 1}(k, \rho(0)) + H_{\nu - 1}(a + n, \rho(a)) \\ \nonumber & = - \nabla^{\nu}_{\rho(0)}H_{\nu - 1}(n, \rho(0)) + H_{\nu - 1}(a + n, \rho(a)) \\ \nonumber & = - H_{\nu - \nu - 1}(n, \rho(0)) + H_{\nu - 1}(a + n, \rho(a)) \quad \text{(By Lemma \ref{Power})}\\ \nonumber & = - 0 + H_{\nu - 1}(a + n, \rho(a)) = H_{\nu - 1}(a + n, \rho(a)).
\end{align}
Hence, by the principle of strong mathematical induction on $n$, the statement \eqref{AB} is true for each $n \in \mathbb{N}_{1}$. Thus, we have \eqref{AA}. From Theorem \ref{T 1} and \eqref{As}, it follows that $$0 \leq |u(t)|  = |u(a)| |\tilde{E}_{c, \nu}(t, a)| \leq |u(a)| \left|H_{\nu - 1}(t, \rho(a))\right| \rightarrow 0 ~ \text{as} ~ t \rightarrow \infty.$$ The proof is complete.
\end{proof}

\section{Conclusion}
Assume $0 < \nu < 1$, $c : \mathbb{N}_{a + 1} \rightarrow \mathbb{R}$. Consider the first order and the $\nu^{\text{th}}$-order nabla difference equations
\begin{align}
& \big{(}\nabla u\big{)}(t) = c(t) u(t), \quad t \in \mathbb{N}_{a + 1}, \label{D 1} \\
& \big{(}\nabla^{\nu}_{\rho(a)} u\big{)}(t) = c(t) u(t), \quad t \in \mathbb{N}_{a + 1}. \label{D 2}
\end{align} 
From \cite{Bo, El}, we know that every solution of \eqref{D 1} tends to $0$ as $t \rightarrow \infty$ if $$\lim_{t \rightarrow \infty}\left|\prod^{t}_{s = a + 1}\left[1 + c(s)\right]\right| = 0.$$ It follows from Theorem \ref{T 2} that every solution of \eqref{D 2} tends to 0 as $t \rightarrow \infty$ if $\left|c(t) + \nu\right| \leq \nu$ for $t \in \mathbb{N}_{a + 1}$. For example, if we take $c(t) = 0$ for $t \in \mathbb{N}_{1}$, every solution of $$\big{(}\nabla u\big{)}(t) = 0, \quad t \in \mathbb{N}_{1},$$ doesn't tends to zero as $t \rightarrow \infty$ whereas every solution of $$\big{(}\nabla^{\nu}_{\rho(0)} u\big{)}(t) = 0, \quad t \in \mathbb{N}_{1},$$ tends to 0 as $t \rightarrow \infty$.


\begin{thebibliography}{99}

\bibitem{Ab 2} Abdeljawad, Thabet; Atıcı, Ferhan M. On the definitions of nabla fractional operators. Abstr. Appl. Anal. 2012 (2012), Art. ID 406757, 13 pp.

\bibitem{Ah} Ahrendt, K.; Castle, L.; Holm, M.; Yochman, K. Laplace transforms for the nabla-difference operator and a fractional variation of parameters formula. Commun. Appl. Anal. 16 (2012), no. 3, 317--347.

\bibitem{At 1} Atıcı, Ferhan M.; Eloe, Paul W., Discrete fractional calculus with the nabla operator. Electron. J. Qual. Theory Differ. Equ. 2009, Special Edition I, no. 3, 12 pp.

\bibitem{At 3} Atıcı, Ferhan M.; Atıcı, Mustafa; Nguyen, Ngoc; Zhoroev, Tilekbek; Koch, Gilbert A study on discrete and discrete fractional pharmacokinetics-pharmacodynamics models for tumor growth and anti-cancer effects. Comput. Math. Biophys. 7 (2019), 10--24.

\bibitem{Bo} Bohner, Martin; Peterson, Allan. Dynamic equations on time scales. An introduction with applications. Birkh\"{a}user Boston, Inc., Boston, MA, 2001.

\bibitem{Ce} \v{C}erm\'{a}k, Jan; Kisela, Tom\'{a}\v{s}; Nechv\'{a}tal, Lud\v{e}k Stability and asymptotic properties of a linear fractional difference equation. Adv. Difference Equ. 2012, 2012:122, 14 pp.

\bibitem{El} Elaydi, Saber An introduction to difference equations. Third edition. Undergraduate Texts in Mathematics. Springer, New York, 2005

\bibitem{Go} Goodrich, Christopher; Peterson, Allan C. Discrete fractional calculus. Springer, Cham, 2015.

\bibitem{Ji} Jia, Baoguo; Erbe, Lynn; Peterson, Allan Comparison theorems and asymptotic behavior of solutions of discrete fractional equations. Electron. J. Qual. Theory Differ. Equ. 2015, Paper No. 89, 18 pp.

\bibitem{Ki} Kilbas, Anatoly A.; Srivastava, Hari M.; Trujillo, Juan J. Theory and applications of fractional differential equations. North-Holland Mathematics Studies, 204. Elsevier Science B.V., Amsterdam, 2006.

\bibitem{Po} Podlubny, Igor. Fractional differential equations. Mathematics in Science and Engineering, 198. Academic Press, Inc., San Diego, CA, 1999.

\bibitem{Wu H} Wu, Hongwu Asymptotic behavior of solutions of fractional nabla difference equations. Appl. Math. Lett. 105 (2020), 106302, 6 pp.
\end{thebibliography}
\end{document}